\documentclass{amsart}
\usepackage{amsmath,amsfonts,amsthm}
\usepackage{a4wide}
\usepackage{graphics}
\usepackage{verbatim}
\usepackage{sagetex}
\usepackage{tikz,tikz-cd,tikz-3dplot}
\usetikzlibrary{matrix,calc,arrows}

\newcommand{\ccomma}{\raisebox{0.4ex}{,}}

\newcommand{\End}{\mathrm{End}}
\newcommand{\FF}{\mathbb{F}}
\newcommand{\MM}{\mathbb{M}}
\newcommand{\QQ}{\mathbb{Q}}
\newcommand{\RR}{\mathbb{R}}

\newcommand{\SO}{\mathrm{SO}}
\newcommand{\SU}[1]{\mathrm{SU}({#1})}
\newcommand{\SUxSU}[1]{\SU{#1}\times\SU{#1}}

\newcommand{\USp}[1]{\mathrm{USp}({#1})}
\newcommand{\ZZ}{\mathbb{Z}}
\newcommand{\disc}{\mathrm{disc}}
\newcommand{\isom}{\cong}

\newcommand{\fD}{\mathfrak{D}}

\newcommand{\cA}{\mathcal{A}}
\newcommand{\cC}{\mathcal{C}}
\newcommand{\cH}{\mathcal{H}}
\newcommand{\cM}{\mathcal{M}}
\newcommand{\cO}{\mathcal{O}}
\newcommand{\cX}{\mathcal{X}}
\newcommand{\eps}{\varepsilon}

\newcommand{\ignore}[1]{}

\newtheorem{theorem}{Theorem}
\newtheorem{corollary}[theorem]{Corollary}
\newtheorem{lemma}[theorem]{Lemma}
\newtheorem{proposition}[theorem]{Proposition}

\title{On Sato--Tate distributions, extremal traces,\\ and real multiplication in genus~2}
\author{David Kohel and Yih-Dar Shieh}
\date{}

\begin{document}

\maketitle
\begin{abstract}
The vertical Sato--Tate conjectures gives expected trace distributions for
for families of curves.  We develop exact expression for the distribution
associated to degree-$4$ representations of $\USp4$, $\SUxSU2$ and $\SU2$
in the neighborhood of the extremities of the Weil bound.
As a consequence we derive qualitative distinctions between the extremal
traces arising from generic genus-$2$ curves and genus-$2$ curves with real
or quaternionic multiplication.  In particular we show, in a specific sense,
to what extent curves with real multiplication dominate the contribution
to extremal traces.
\end{abstract}

\input{code/sage_input.sage}

%%%%%%%%%%%%%%%%%%%%%%%%%%%%%%%%%%%%%%%%%%%%%%%%%%%%%%%%%%%%%%%%%%%%%%%%%%%%%%%%
%%%%%%%%%%%%%%%%%%%%%%%%%%%%%%%%%%%%%%%%%%%%%%%%%%%%%%%%%%%%%%%%%%%%%%%%%%%%%%%%
\section{Introduction}
%%%%%%%%%%%%%%%%%%%%%%%%%%%%%%%%%%%%%%%%%%%%%%%%%%%%%%%%%%%%%%%%%%%%%%%%%%%%%%%%
%%%%%%%%%%%%%%%%%%%%%%%%%%%%%%%%%%%%%%%%%%%%%%%%%%%%%%%%%%%%%%%%%%%%%%%%%%%%%%%%

Let $C/\FF_q$ be a curve of genus $g$ over the field $\FF_q$ of $q$ elements
and let $J$ be its Jacobian.  The characteristic polynomial of the Frobenius
endomorphism $\pi \in \End(J)$, is of the form
$$
\chi_C(x) = x^{2g} - a_1 x^{2g-1} + a_2 x^{2g-2} + \cdots
+ a_2 q^{g-2} x^2 - a_1 q^{g-1} x + q^g.
$$
We define the normalized Weil polynomial to be the characteristic polynomial
of $\pi \otimes \sqrt{q}^{-1}$ in $\End(J) \otimes_\ZZ \RR$:
$$
\tilde{\chi}_C(x) = x^{2g} - \tilde{a}_1 x^{2g-1} + \tilde{a}_2 x^{2g-2} + \cdots + \tilde{a}_2 x^2 - \tilde{a}_1 x + 1,
$$
with normalized roots $\tilde{\alpha}_1,\dots,\tilde{\alpha}_g,\tilde{\alpha}_{g+1},\dots,\tilde{\alpha}_{2g}$
satisfying
$
\tilde{\alpha}_j\tilde{\alpha}_{j+g} = 1 \mbox{ where } \tilde{\alpha}_j = e^{i\theta_j}.
$
We denote the real numbers $\tilde{\alpha}_j + \tilde{\alpha}_{j+g}$ by $t_j$, and define $(s_1,\dots,s_g)$
to be the symmetric polynomials in $(t_1,\dots,t_g)$.  We call the polynomial
$$
\prod_{i=1}^g (x - t_i) = x^g - s_1 x^{g-1} + s_2 x^{g-2} + \cdots + (-1)^g s_g
$$
the normalized real Weil polynomial of $C/\FF_q$, noting in particular that $s_1 = \tilde{a}_1$
and $s_2 = \tilde{a}_2 - g$.

The Sato--Tate conjecture for an non-CM elliptic curve $C/\QQ$ concerns the equidistribution
with respect to a Haar measure induced by the group $\SU2$ of the Frobenius angles $\theta_1$
or normalized Frobenius traces $t_1$ as $q$ varies over the reductions of $C$ over finite
fields $\FF_q$, and has been generalized to curves of higher genus or abelian varieties of
higher dimension. In higher degree, this is expressed as an equidistribution of the tuples
of Frobenius angles $(\theta_1,\dots,\theta_g)$, of real traces $(t_1,\dots,t_g)$ or of
the symmetric sums $(s_1,\dots,s_g)$ with respect to a induced Haar measure of an
associated compact Lie group, the Sato--Tate group.

These generalizations are typically grouped together as horizontal Sato--Tate conjectures.
Katz introduced a vertical Sato--Tate conjecture concerning the variation
of a family $\cC/S$ of curves, abelian varieties or exponential sums over a base scheme $S$.
Specifically, the vertical Sato--Tate conjectures concerns the limit of the sets of
Frobenius traces associated to the fibers point in $S(\FF_q)$, as $q$ goes to infinity.
One can view the horizontal Sato--Tate conjecture as a statement that Frobenius distribution
for a fiber over a point $S$ follows the equidistribution of the vertical Sato--Tate
distribution.  Whereas the horizontal Sato--Tate conjectures are, in many instances, still
open, the vertical Sato--Tate conjectures are more amenable to proof.
In this work we assume the expected vertical Sato--Tate conjecture and as a consequence
derive a qualitative distinction between extremal traces for quaternionic multiplication (QM),
real multiplication (RM) and generic families,
and when refering to the Sato--Tate group over a finite field we are implicitly
considering the vertical aspect of these conjectures.

In what follows we consider the case of genus-$2$ curves, for which the generic Sato--Tate
group is $\USp4$.  When the Jacobian is split or admits RM over
the base field, the Frobenius endomorphism commutes with this structure, and the
Sato--Tate group is restricted to $\SUxSU2$.  Such families of curves arise when the
base is a cover of a Humbert surface $\cH_D$ over which the RM ring is defined.
The Humbert surface is a moduli space, contained in the $3$-fold $\cM_2$ of moduli
for genus-$2$ curves, classifying isomorphism classes of curves whose Jacobians
admit endomorphisms by a real quadratic order of discriminant $D$.  A special case
is when $D$ is a square, corresponding to orders in the real quadratic ring
$\ZZ[x]/(x^2-x) \isom \ZZ\times\ZZ$ of discriminant $1$.
The embedding of the Sato--Tate group $\SUxSU2$ in $\USp4$ depends on the choice
of Humbert surface (in particular on its discriminant), but the induced Haar measure
on Frobenius traces remains invariant of this embedding.  Over a finite field,
every curve $C/\FF_q$ lies in the image of some Humbert surface, and the objective
of this work is to understand how this stratification of the space $\cM_2$ by
Humbert surfaces $\cH_D$ lets us understand the distributions of normalized Frobenius
traces $s_1$ near the extremities of the Weil interval $[-2g,2g] = [-4,4]$.
To complete the picture of this stratification, we consider the Frobenius
distributions of families with quaternionic multiplication (QM).  These arise
as families over Shimura curves $X^N$, admitting QM by an indefinite order
in a quaternion algebra of discriminant $N$.  Such curves appear as components
of the intersections of Humbert surfaces, and give rise to the Sato--Tate group
$\SU2$ embedded diagonally (up to conjugation) as a subgroup
$\Delta \subset \SUxSU2$ in a degree-$4$ representation.

The study of extremal traces (in particular curves with many points) has a long
history, motivated by applications to coding theory and the rich mathematical
structure going into their study.  In the next section we introduce the
notation for this study before turning to the Weyl integration formulas in
Section~\ref{section:weyl_integration}, which gives an explicit form for the
Haar measures induced by the symplectic groups $\SU2$ and $\USp4$.
In the following Section~\ref{section:taylor_expansions}, we develop precise
Taylor series expressions for the Haar measure induced on the trace function
$s_1$, and devote Section~\ref{section:lachaud} to certain global expressions
for the continuous distribution functions for $\USp4$ and $\SUxSU2$ on $s_1$
obtained by Gilles Lachaud.  As an application to the explicit Taylor series
expansions, in Section~\ref{section:extremal_traces}, we give a qualitative
comparison of the expected contribution of a generic family over $\cM_2$,
of an RM family over some Humbert surface $\cH_D$, and of a QM family over
a Shimura curve $X^N$.  In particular we show that a Shimura curve $X^N$
and associated Sato--Tate group $\Delta$ gives the greatest density of extremal
traces of these groups, but that a Humbert surface of small discriminant
dominates the contribution to extremal traces of genus-$2$ curves.
\vspace{2mm}

\noindent{\bf Acknowledgements.}
This work, presented in a previous instance of AGCT in 2015, grew out of
discussions in the course of the doctoral work of the second author, jointly
supervised by Gilles Lachaud and the first author.  Certain results of Gilles
in the direction of this work are presented in Section~\ref{section:lachaud}.
Gilles' departure was both a deep personal and mathematical loss, and his
impact continues to be felt in work such the present article and through
the international AGCT workshops.

%%%%%%%%%%%%%%%%%%%%%%%%%%%%%%%%%%%%%%%%%%%%%%%%%%%%%%%%%%%%%%%%%%%%%%%%%%%%%%%%
%%%%%%%%%%%%%%%%%%%%%%%%%%%%%%%%%%%%%%%%%%%%%%%%%%%%%%%%%%%%%%%%%%%%%%%%%%%%%%%%
\section{Background and notation}
\label{section:background}
%%%%%%%%%%%%%%%%%%%%%%%%%%%%%%%%%%%%%%%%%%%%%%%%%%%%%%%%%%%%%%%%%%%%%%%%%%%%%%%%
%%%%%%%%%%%%%%%%%%%%%%%%%%%%%%%%%%%%%%%%%%%%%%%%%%%%%%%%%%%%%%%%%%%%%%%%%%%%%%%%

We consider Galois representations in $\USp4$ and its subgroups $\SUxSU2$,
and $\SU2$ arising from families of curves of genus~$2$.  We assume a generic
normalized Weil polynomial takes the form
$$
\widetilde\chi(x) = x^4 - s_1 x^3 + (s_2+2) x^2 - s_1 x + 1,
$$
in particular, $s_1$ represents the normalized trace of a general element.
Under a splitting of the generic normalized Weil polynomial as
$$
(x - e^{i\theta_1})(x - e^{-i\theta_1})(x - e^{i\theta_2})(x - e^{-i\theta_2})
= (x^2 - t_1 x + 1)(x^2 - t_2 x + 1),
$$
we consider the transformation of Haar measures induced by the changes
of variables between local parameters $(\theta_1,\theta_2)$, $(t_1,t_2)$,
and $(s_1,s_2)$ on the respective domains of support $\Theta_2$, $I_2$
and $\Sigma_2$.
In order to focus on the role of the discriminant of the normalized real
polynomial:
$$
(x - 2\cos(\theta_1))(x - 2\cos(\theta_2)) = (x - t_1)(x - t_2) = x^2 - s_1 x + s_2,
$$
we consider $D_0 = (t_1-t_2)^2 = s_1^2 - 4s_2$, and its square root
$\delta_0 = t_2 - t_1$.  Similarly, for a root $\pi$ of $\chi(x)$,
with conjugate $\bar{\pi} = \pi^{-1}$, the relative ring extension
$\ZZ[\pi,\bar{\pi}]/\ZZ[\pi+\bar{\pi}]$ has norm discriminant
$$
D_1 = (4 - t_1^2)(4 - t_2^2)
    = (4 + s_2)^2 - 4s_1^2
    = (4 - 2s_1 + s_2)(4 + 2s_1 + s_2).
$$
As in the above construction, we fix the relations such as $(s_1,s_2)
= (t_1+t_2,t_1t_2)$, and identify $D_0$ and $D_1$ in $\ZZ[t_1,t_2]$
and $\ZZ[s_1,s_2]$, irrespective of the ring parameters over which
they are defined.

The goal of this work is to investigate the asymptotic trace distributions
in the neighborhood of $s_1 = -4$ (or by symmetry, in the neighborhood of
$s_1 = 4$). In particular we investigate the interpretation of the Haar
measure on conjugacy classes via a stratification of the moduli space $\cM_2$,
with generic Sato--Tate group $\USp4$, by the Humbert surfaces $\cH_D$, with
generic Sato--Tate $\SUxSU2$, and the contributions from Shimura
curves, whose associated Sato--Tate group is a diagonal image of $\SU2$.
We recall that the embeddings of these groups determined by the particular
Humbert surface or Shimura curve varies by a conjugation, but the traces
and integration formulas remain invariant.

%%%%%%%%%%%%%%%%%%%%%%%%%%%%%%%%%%%%%%%%%%%%%%%%%%%%%%%%%%%%%%%%%%%%%%%%%%%%%%%%
%%%%%%%%%%%%%%%%%%%%%%%%%%%%%%%%%%%%%%%%%%%%%%%%%%%%%%%%%%%%%%%%%%%%%%%%%%%%%%%%
\section{Weyl integration formula}
\label{section:weyl_integration}
%%%%%%%%%%%%%%%%%%%%%%%%%%%%%%%%%%%%%%%%%%%%%%%%%%%%%%%%%%%%%%%%%%%%%%%%%%%%%%%%
%%%%%%%%%%%%%%%%%%%%%%%%%%%%%%%%%%%%%%%%%%%%%%%%%%%%%%%%%%%%%%%%%%%%%%%%%%%%%%%%

The Weyl integration formula gives an integral expression for the Haar measure
on the space of Frobenius angles $(\theta_1,\dots,\theta_g) \in \Theta_g =
[0,\pi]^g$ or of the space of real traces:
$$
(t_1,\dots,t_g) = (2\cos(\theta_1),\dots,2\cos(\theta_g)) \in I_g = [-2,2]^g.
$$
In particular we focus on $g = 2$ and the distributions induced by subgroups
of $\USp4$.  We recall the form of the Weyl integration formula for the unitary
symplectic group from Weyl~\cite[Theorem 7.8B]{Weyl} (see Katz and Sarnak
\cite[\S5.0.4]{Katz-Sarnak}).

\begin{theorem}[Weyl]
The Haar measure induced by $\USp{2g}$ on the angle space $\Theta_g$ is given
by the formula
$$
\mu_G(\theta_1,\dots,\theta_g) = \frac{2^{g^2}}{g!\pi^g} \prod_{i<j} \big(\cos(\theta_i) - \cos(\theta_j)\big)^2
\prod_{i=1}^g \sin^2(\theta_i) \,d\theta_1 \cdots d\theta_g.
$$
\end{theorem}

\subsection*{Degree 2}
As a corollary, we specialize to the case $g = 1$, where we find the classic distributions
for the degree-$2$ representations of $\SU2 = \USp2$ arising in the Sato--Tate conjecture
for elliptic curves.  To simplify notation we note $K = \SU2$ and use the same notation
$\mu_K$ for the measure on the spaces $\Theta = [0,\pi]$ and $I = [-2,2]$, distiguishing
the domain by the variable name.

\begin{corollary}%[$\SU2$]
\label{corollary:haar_su2}
The Haar measure induced by $K = \SU2$ on the angle space $\Theta = [0,\pi]$ is given by
the formula
$$
\mu_K(\theta) = \frac{2}{\pi} \sin^2(\theta) d\theta,
$$
and, in terms of the trace $t = 2\cos(\theta)$ in $I = [-2,2]$, the Haar measure takes the form
$$
\mu_K(t) = \frac{1}{2\pi} \sqrt{4-t^2}\,dt.
$$
\end{corollary}

We define the associated probability density functions on $\Theta$ and $I$ by
$$
f_K(\theta) = \frac{\mu_K(\theta)}{d\theta} = \frac{2}{\pi}\sin^2(\theta) \mbox{ and }
f_K(t) = \frac{\mu_K(t)}{dt} = \frac{1}{2\pi}\sqrt{4-t^2}.
$$
giving well-known density functions of the Sato--Tate conjectures:
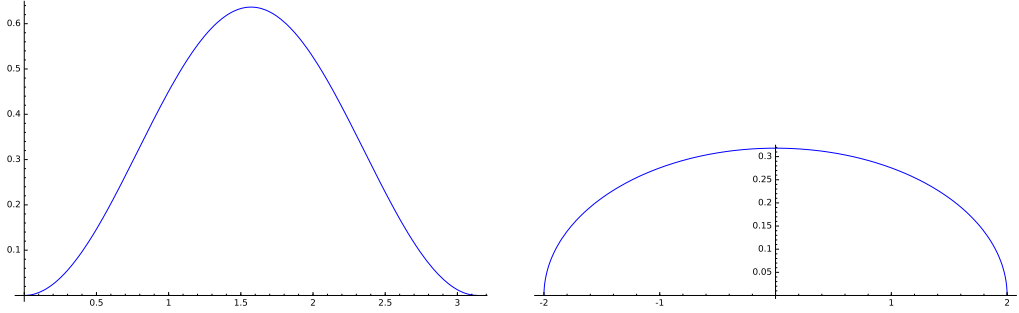
\begin{center}
\scalebox{0.33}{\sageplot{f_SU2_theta.plot(x,0,pi,aspect_ratio=3.14)}}\quad
\scalebox{0.33}{\sageplot{f_SU2_trace.plot(x,-2,2,aspect_ratio=4)}}
\end{center}
and associated cumulative density functions
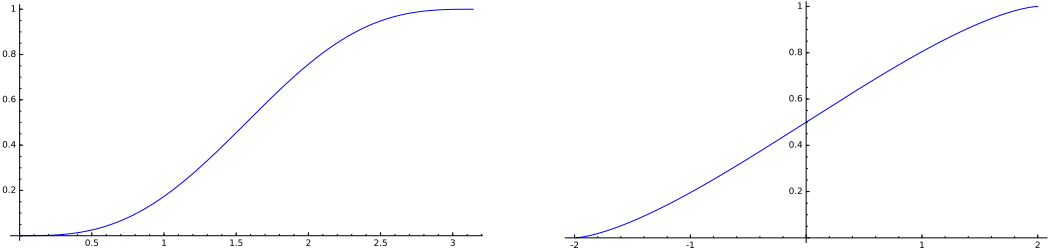
\begin{center}
\begin{tabular}{c}
\scalebox{0.33}{\sageplot{F_SU2_theta.plot(x,0,pi,aspect_ratio=1.57)}}\\
$\displaystyle F_K(\theta) = \int_0^\theta \mu_K(\theta)$
\end{tabular}
\quad
\begin{tabular}{c}
\scalebox{0.33}{\sageplot{F_SU2_trace.plot(x,-2,2,aspect_ratio=2)}}\\
$F_K(t) = \displaystyle\int_{-2}^t \mu_K(t)$
\end{tabular}
\end{center}
In particular, the latter distribution, on the trace space $I$, the cumulative
distribution function measures the contribution of extremal traces in a
neighborhood of $t = -2$, corresponding to elliptic curves with many points.

\subsection*{Degree 4}

From the perspective of Sato--Tate distributions of genus-$2$ curves, we are interested
in the degree-$4$ representations of $\USp4$, of a subgroup $\SUxSU2$, and
of the image of the diagonal map $\Delta: \SU2 \rightarrow \SUxSU2$.

\begin{corollary}%[$\USp4$]
\label{corollary:haar_usp4_theta}
The Haar measure induced by $G = \USp4$ on the angle space $\Theta_2$ is given by the formula
$$
\mu_G(\theta_1,\theta_2) =
  \frac{8}{\pi^2}
    \left(\cos(\theta_1) - \cos(\theta_2)\right)^2
    \sin^2(\theta_1)\sin^2(\theta_2) d\theta_1 d\theta_2.
$$
\end{corollary}

\noindent
The Haar measure for $\SUxSU2$ is obtained as the product measure for the case $g = 1$.

\begin{corollary}%[$\SUxSU2$]
\label{corollary:haar_su2xsu2_theta}
The Haar measure induced by $H = \SUxSU2$ on the angle space $\Theta_2$ is given by the formula
$$
\mu_H(\theta_1,\theta_2) = \frac{4}{\pi^2} \sin^2(\theta_1)\sin^2(\theta_2) d\theta_1d\theta_2.
$$
\end{corollary}

In what follows we continue to use the notation $G$ for $\USp4$, $H$ for $\SUxSU2$, and
write $\Delta$ for the diagonal image of $\SU2$ in $H$.
Making the change of variables $t_j = 2\cos\theta_j$, such that
$
\sin(\theta_1)\sin(\theta_2) d\theta_1d\theta_2 = dt_1 dt_2,
$
we obtain the expression in $(t_1,t_2)$ for the Haar measure induced by $G$ and $H$:
$$
\begin{array}{r@{\;}c@{\;}l}
\mu_G(t_1,t_2) & = & \displaystyle
       \frac{1}{8\pi^2} (t_1-t_2)^2 \sqrt{(4-t_1^2)(4-t_2^2)}\, dt_1 dt_2
     = \frac{1}{8\pi^2} D_0 \sqrt{D_1}\, dt_1 dt_2,\\[4mm]
\mu_H(t_1,t_2) & = & \displaystyle
       \frac{1}{4\pi^2} \sqrt{(4-t_1^2)(4-t_2^2)}\, dt_1 dt_2
     = \frac{1}{4\pi^2}\sqrt{D_1}\,dt_1 dt_2.
\end{array}
$$
where $D_0 = (t_1-t_2)^2$ and $D_1 = (4 - t_1^2)(4 - t_2^2)$.
These groups of rank $2$ have measure supported on the whole domain $I_2$.
The measure for $\Delta$ is supported on the closed domain $t_1 = t_2$
(or $\theta_1 = \theta_2$ in $\Theta_2$).
We defer the discussion of the trace distribution for $\Delta$ until after
the introduction of suitable transformations of the domain.

\subsection*{Domains of integration}

In order to focus on the behavior of the normalized trace function $s_1 = t_1 + t_2$
of the representation in $\USp4$, and relate the measure to the coefficients of the
normalized Weil polynomial, we first consider the transformation $I_2 \rightarrow
\Sigma_2$, where
$$
\Sigma_2 = \{ (s_1,s_2) \in \RR^2 \;|\; 2|s_1| \le s_2 + 4, 4s_2 \le s_1^2 \},
$$
sending $(t_1,t_2)$ to $(s_1,s_2) = (t_1+t_2,t_1t_2)$.
$$
%\scalebox{0.33}{\sageplot{I2}} \quad\raisebox{23.5mm}{$\longrightarrow$}\quad \scalebox{0.33}{\sageplot{X2}}
\scalebox{0.25}{\sageplot{I2}} \quad\raisebox{17.625mm}{$\longrightarrow$}\quad \scalebox{0.25}{\sageplot{X2}}
$$

As for the discriminants $D_i$, we continue to use the same names for the measures
$\mu_G$ and $\mu_H$ on the spaces $\Theta_2$, $I_2$ and $\Sigma_2$, distinguishing
the domain by the variables names.  In particular, we recall the previously
determined expressions for the discriminants
$$
D_0 = s_1^2 - 4s_2 \mbox{ and } D_1 = (4 + s_2)^2 - 4s_1^2
$$
and, in view of the equivalence of alternating volume forms,
$$
ds_1 ds_2 = (dt_1 + dt_2)(t_1 dt_2 + t_2 dt_1) = (t_1 - t_2) dt_1 dt_2 = \sqrt{D_0}\,dt_1 dt_2,
$$
and a factor of 2 from the double cover $I_2 \rightarrow \Sigma_2$, the induced Haar measure becomes
$$
\mu_G(s_1,s_2) = \frac{1}{4\pi^2} \sqrt{D_0 D_1}\, ds_1 ds_2, \mbox{ and }
\mu_H(s_1,s_2) = \frac{1}{2\pi^2} \sqrt{\frac{D_1}{D_0}}\, ds_1 ds_2.
$$

In order to identify the role of the normalized discriminant $D_0$ of the real subring
$\ZZ[\pi+\bar{\pi}]$ and with the view of obtaining a simple domain of integration for
integrating over the fibers above $s_1$, we set $\delta_0^2 = D_0 = s_1^2 - 4 s_2$
and apply the transformation $(s_1,s_2) \mapsto (s_1,\delta_0)$ from $\Sigma_2$
to the domain
$$
\fD_2 = \left\{ (s_1,\delta_0) \in [-4,4]\times[0,4] \;|\; \delta_0 \pm s_1 \le 4 \right\}.
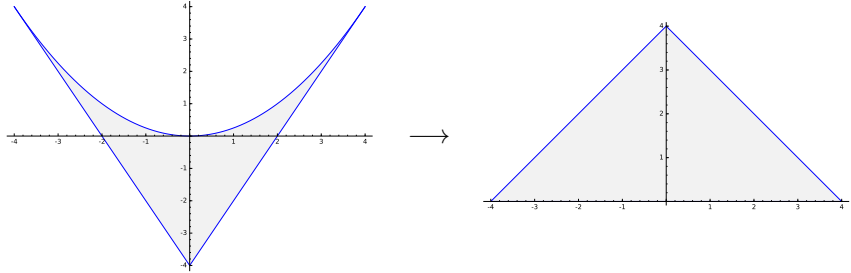
$$
$$
%\scalebox{0.33}{\sageplot{X2}} \quad\raisebox{23.5mm}{$\longrightarrow$}\quad
%\raisebox{8mm}{\scalebox{0.33}{\sageplot{D2}}}
\scalebox{0.25}{\sageplot{X2}} \quad\raisebox{17.625mm}{$\longrightarrow$}\quad
\raisebox{8mm}{\scalebox{0.25}{\sageplot{D2}}}
$$
By taking the positive branch of $\delta_0 = \sqrt{D_0} \ge 0$, this map is an isomorphism.
Moreover, from $s_1^2 - 4s_2 = \delta_0^2$, the volume forms satisfy
$$
ds_1ds_2 = \frac{1}{2} \delta_0\, d\delta_0 ds_1,
$$
from which we write
$$
\mu_G(s_1,\delta_0) = \frac{1}{8\pi^2} \delta_0^2 \sqrt{D_1}\, d\delta_0 ds_1, \mbox{ and }
\mu_H(s_1,\delta_0) = \frac{1}{4\pi^2} \sqrt{D_1}\, d\delta_0 ds_1,
$$
noting that $D_1$ takes the form 
$
\displaystyle
D_1 = \frac{((4 + s_1)^2 - \delta_0^2)((4 - s_1)^2 - \delta_0^2)}{16}\cdot
$
\vspace{2mm}

\noindent{\bf Remark.}
The formal agreement of the expressions for $\mu_G$ on $I_2$ and $\fD_2$,
$$
\mu_G(t_1,t_2) = \frac{1}{8\pi^2} D_0 \sqrt{D_1}\,dt_1 dt_2
\mbox{ and }
\mu_G(s_1,\delta_0) = \frac{1}{8\pi^2} D_0 \sqrt{D_1}\,d\delta_0 ds_1,
$$
and similarly for $\mu_H$,
$$
\mu_H(t_1,t_2) = \frac{1}{4\pi^2}\sqrt{D_1}\,dt_1 dt_2 \mbox{ and }
\mu_H(s_1,\delta_0) = \frac{1}{4\pi^2} \sqrt{D_1}\, d\delta_0 ds_1,
$$
reflects the fact that $(s_1,\delta_0) = (t_1+t_2,t_1-t_2)$ is a transformation
of determinant~$2$, so $d\delta_0ds_1 = 2dt_1dt_2$, while the domain of
integration $\fD_2$ consists of points $(s_1,\delta_0)$ in the upper half of
the image of points $(t_1,t_2) \in I_2$, such that $\delta_0 = t_1 - t_2 \ge 0$.

In contrast to the above rank-$2$ groups, the group $\Delta$ is a rank-1 group,
for which the support of the Haar measure is restricted to the $1$-dimensional
subspace $\delta_0 = t_1 - t_2 = 0$.
$$
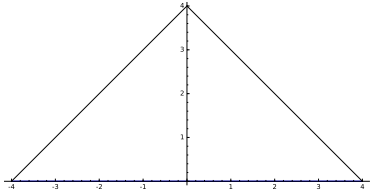
\scalebox{0.25}{\sageplot{D2_zero}}
$$
In fact, pulling back the measure for $\SU2$ by $s_1 = t_1 + t_2 = 2t$ in
Corollary~\ref{corollary:haar_su2}, we obtain the induced Haar measure in 
terms of the trace.

\begin{corollary}
\label{corollary:haar_delta}
The Haar measure induced by $\Delta$ on the trace $s_1 \in [-4,4]$ is given by
$$
\mu_\Delta(s_1) = \frac{1}{8\pi}\sqrt{16 - s_1^2}\,ds_1.
$$
\end{corollary}
\noindent
This measure can be viewed as a measure along the subdomain $[-4,4] \times \{0\}
\subset \fD_2$, or as a product distribution on $\fD_2$ such that the measure
in $\delta_0$ is a Dirac delta function with density $1$ on $\delta_0 = 0$.

In what follows we seek to determine analogous expressions for the induced
measure on the trace function $s_1$ in $[-4,4]$, coming from the rank-$2$
groups $G$ and $H$ by integrating in $\delta_0$ along the fibers about $s_1$.
In the next section, we carry this out to develop exact series approximations
for the measures in the neighborhood of $s_1=-4$.

%%%%%%%%%%%%%%%%%%%%%%%%%%%%%%%%%%%%%%%%%%%%%%%%%%%%%%%%%%%%%%%%%%%%%%%%%%%%%%%%
%%%%%%%%%%%%%%%%%%%%%%%%%%%%%%%%%%%%%%%%%%%%%%%%%%%%%%%%%%%%%%%%%%%%%%%%%%%%%%%%
\section{Taylor expansions for the trace function}
\label{section:taylor_expansions}
%%%%%%%%%%%%%%%%%%%%%%%%%%%%%%%%%%%%%%%%%%%%%%%%%%%%%%%%%%%%%%%%%%%%%%%%%%%%%%%%
%%%%%%%%%%%%%%%%%%%%%%%%%%%%%%%%%%%%%%%%%%%%%%%%%%%%%%%%%%%%%%%%%%%%%%%%%%%%%%%%

We now introduce a transformation which permits us to develop a series expansion
for the Haar measure of the trace $s_1$ in the neighborhood of the endpoint
$s_1 = -4$. We first define the space $\Lambda_2 = [0,4] \times [0,1]$ of points
$(\eps,\lambda)$ and a map $\Lambda_2 \rightarrow \fD_2$ given by
$$
(\eps,\lambda) \longmapsto (\eps-4,\eps\lambda).
$$
This gives a parametrization of the left-half subspace of $\fD_2$.
$$
%\scalebox{0.33}{\sageplot{L2}} \quad\raisebox{23.5mm}{$\longrightarrow$}\quad
%\raisebox{8mm}{\scalebox{0.33}{\sageplot{D2}}}
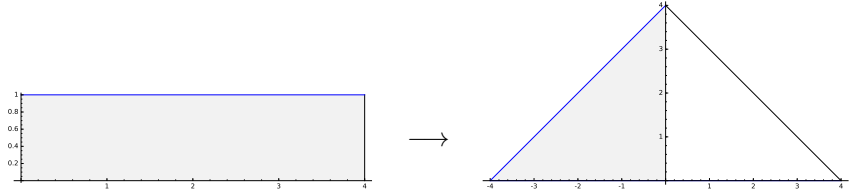
\scalebox{0.25}{\sageplot{L2}} \quad\raisebox{6.50mm}{$\longrightarrow$}\quad
\scalebox{0.25}{\sageplot{D2_half}}
$$
On the space $\Lambda_2$, we find the following expression for $D_1$:
$$
D_1 = \frac{(\eps^2-\delta_0^2)((8-\eps)^2-\delta_0^2)}{16}
    = \frac{\eps^2(1-\lambda^2)((8-\eps)^2-\eps^2\lambda^2)}{16}\ccomma
$$
and setting $\rho = \eps/(8-\eps)$ gives
$$
\sqrt{D_1} = \frac{\eps(8-\eps)}{4}\sqrt{(1-\lambda^2)(1-\rho^2 \lambda^2)},
$$
Using $(s_1,\delta) = (\eps-4,\eps\lambda)$, it follows that
$
d\delta_0ds_1 = (\eps d\lambda + \lambda d\eps)d\eps = \eps d\lambda d\eps,
$
and we find
%$$
\begin{equation}
\label{equation:mu_epsilon-lambda_G}
\mu_G(\eps,\lambda) =
    \frac{\eps^4(8-\eps)}{32\pi^2}
    \left(\lambda^2 \sqrt{(1-\lambda^2)(1-\rho^2 \lambda^2)}
    \, d\lambda\right) d\eps,
\end{equation}
%$$
and
%$$
\begin{equation}
\label{equation:mu_epsilon-lambda_H}
\mu_H(\eps,\lambda) = \frac{\eps^2(8-\eps)}{16\pi^2}
  \left(\sqrt{(1-\lambda^2)(1-\rho^2 \lambda^2)}\,d\lambda\right) d\eps.
\end{equation}
%$$
In a neighborhood of $\eps = 0$, we have $\rho = \eps/8 + O(\eps^2)$, and
therefore $1 - \rho^2\lambda^2 = 1 + O(\eps^2)$, giving
%$$
\begin{equation}
\label{equation:mu_epsilon-lambda_approx}
\mu_G(\eps,\lambda) \approx \frac{\eps^4(8-\eps)}{32\pi^2} \left(\lambda^2 \sqrt{1-\lambda^2}\,d\lambda\right) d\eps
%$$
\mbox{ and }
%$$
\mu_H(\eps,\lambda) \approx \frac{\eps^2(8-\eps)}{16\pi^2}
  \left(\sqrt{1-\lambda^2}\,d\lambda\right) d\eps.
\end{equation}
%$$
From the values of the inner integrals,
$$
\int_0^1 \lambda^2 \sqrt{1-\lambda^2}\, d\lambda = \frac{\pi}{16} \mbox{\ \ and }
\int_0^1 \sqrt{1-\lambda^2}\, d\lambda = \frac{\pi}{4}\ccomma
$$
we obtain the trace distributions near $\eps = 0$.
\begin{theorem}
\label{theorem:haar_GH_eps_approx}
The Haar measures associated to the trace function on the groups $G = \USp4$ and
$H = \SUxSU2$ have the following approximations at $\eps = s_1 + 4 = 0$.
%$$
\begin{equation}
\mu_G(\eps) = \left(\frac{\eps^4(8-\eps)}{512\pi} + O(\eps^6)\right) d\eps
  = \left(\frac{\eps^4}{64\pi} - \frac{\eps^5}{512\pi} + O(\eps^6)\right) d\eps
\end{equation}
%$$
and
%$$
\begin{equation}
\mu_H(\eps) = \left(\frac{\eps^2(8-\eps)}{64\pi} + O(\eps^4)\right) d\eps
  = \left(\frac{\eps^2}{8\pi} - \frac{\eps^3}{64\pi} + O(\eps^4)\right) d\eps.
\end{equation}
%$$
\end{theorem}

In order to simplify compute the full Taylor expansions, we begin by recalling
the definition of the Catalan numbers
$$
C_n = \frac{1}{n+1}{2n \choose n} = {2n \choose n} - {2n \choose n + 1}\cdot
$$
noting that in particular $C_0 = C_1 = 1$. With this definition we can state
a lemma concerning the form of a class of integrals needed for our series
approximation.
\begin{lemma}
\label{lemma:catalan_integrals}
$$
\int_0^1 \lambda^{2n} \sqrt{1-\lambda^2}\, d\lambda = \frac{\pi}{4}\cdot\frac{C_n}{4^n}\cdot
$$
\end{lemma}
\noindent
We finally recall the form of a Taylor series for the square root.
\begin{lemma}
\label{lemma:catalan_sqrt}
The Taylor expansion for $\sqrt{1 - x}$ is given by the power series
$$
\sqrt{1 - x} = 1 - \frac{x}{2} \sum_{n=0}^\infty \frac{C_n} {4^n} x^n.
$$
\end{lemma}

\noindent
Applying Lemma~\ref{lemma:catalan_sqrt} to $\sqrt{1 - \rho^2\lambda^2}$, followed
by Lemma~\ref{lemma:catalan_integrals} to the resulting integral summands, we
obtain exact Taylor series expansions for the integrals
%$$
\begin{equation}
\label{equation:integral_G}
\begin{array}{r@{}l}
\displaystyle
\int_0^1 & \lambda^2 \sqrt{(1 - \lambda^2)(1 - \rho^2\lambda^2)}\,d\lambda\\
& = \displaystyle
  \int_0^1 \lambda^2 \sqrt{(1 - \lambda^2)}\,d\lambda
  \ -\ \frac{\rho^2}{2} \sum_{n=0}^\infty \frac{C_n}{4^n}\left(
    \int_0^1 \lambda^{2(n+2)}\sqrt{1 - \lambda^2}\,d\lambda\right)\rho^{2n}\\[4mm]
& = \displaystyle
  \frac{\pi}{16} \left( 1 -  \frac{\rho^2}{8} \sum_{n=0}^\infty \frac{C_nC_{n+2}}{16^n}\rho^{2n} \right)\!\ccomma
\end{array}
\end{equation}
%$$
and
%$$
\begin{equation}
\label{equation:integral_H}
\begin{array}{r@{}l}
\displaystyle
\int_0^1 & \sqrt{(1 - \lambda^2)(1 - \rho^2\lambda^2)}\,d\lambda\\
& = \displaystyle
  \int_0^1 \sqrt{(1 - \lambda^2)}\,d\lambda
  \ -\ \frac{\rho^2}{2} \sum_{n=0}^\infty \frac{C_n}{4^n}\left(
    \int_0^1 \lambda^{2(n+1)}\sqrt{1 - \lambda^2}\,d\lambda\right)\rho^{2n}\\[4mm]
& = \displaystyle
  \frac{\pi}{4} \left( 1 -  \frac{\rho^2}{8} \sum_{n=0}^\infty \frac{C_nC_{n+1}}{16^n}\rho^{2n} \right)\!\cdot
\end{array}
\end{equation}
%$$
Substituting equations~\eqref{equation:integral_G} and~\eqref{equation:integral_H}
into equations \eqref{equation:mu_epsilon-lambda_G} and~\eqref{equation:mu_epsilon-lambda_H}
for the Haar measures arising from $G = \USp4$ and $H = \SUxSU2$, gives
the following theorem.
\begin{theorem}
\label{theorem:haar_GH_eps}
With the notation as above, the Taylor expansions for the trace on $G$ and $H$,
in the neighborhood of $\eps = s_1 + 4 = 0$ are given by
%$$
\begin{equation}
\mu_G(\eps) = \int_{\lambda=0}^1 \mu_G(\eps,\lambda) = \frac{\eps^4(8-\eps)}{512\pi}
    \left( 1 -  \frac{\rho^2}{8} \sum_{n=0}^\infty \frac{C_nC_{n+2}}{16^n}\rho^{2n} \right) d\eps,
\end{equation}
%$$
and
%$$
\begin{equation}
\mu_H(\eps) = \int_0^1 \mu_H(\eps,\lambda) = \frac{\eps^2(8-\eps)}{64\pi}
   \left( 1 -  \frac{\rho^2}{8} \sum_{n=0}^\infty \frac{C_nC_{n+1}}{16^n}\rho^{2n} \right) d\eps.
\end{equation}
%$$
\end{theorem}

\noindent{\bf Remark.} As a consequence of Stirling's formula, the Catalan numbers
are known to satisfy the asymptotic growth
$$
\frac{C_n}{4^n} \sim \frac{1}{\sqrt{\pi}\,n^{3/2}}\ccomma
$$
from which we can conclude the convergence of the above formulas for $\eps$ in
the interval $[0,4]$, since $0 \le \rho = \eps/(8-\eps) \le 1$. This implies
convergence for the trace $s_1$ in the interval $[-4,0]$, and by symmetry around
$s_1 = 0$, extending to the interval $[-4,4]$.
\vspace{2mm}

We conclude this section with the analogous series expansion for the degree-$4$
diagonal subgroup $\Delta$.

\begin{theorem}
\label{theorem:haar_delta_eps}
The Haar measure induced by $\Delta$ in the neighborhood of $\eps = s_1 + 4 = 0$
is given by
$$
\mu_\Delta(\eps)
  = \frac{1}{\sqrt{8}\,\pi} \sqrt{\eps(1 - \eps/8)}
  = \frac{\eps^{1/2}}{\sqrt{8}\,\pi} \left(
  1\ - \ \frac{\eps}{16}\sum_{n=0}^\infty \frac{C_n}{4^n} \left(\frac{\eps}{8}\right)^n \right)\!\cdot
$$
\end{theorem}

\begin{proof}
Substituting $\eps = s_1 + 4$ in Corollory~\ref{corollary:haar_delta} gives
$$
\mu_\Delta(\eps)
   = \frac{1}{8\pi}\sqrt{\eps(8-\eps)}\,ds_1
   = \frac{\eps^{1/2}}{\sqrt{8}\pi}\sqrt{1-\eps/8}\,ds_1.
$$
Applying Lemma~\ref{lemma:catalan_sqrt}, we obtain the desired expansion.
\end{proof}

\noindent
By the previous remark, this gives convergence for $\eps$ in $[0,8]$ and
thus $s_1$ in $[-4,4]$.
\vspace{2mm}

We recall that the density function of the trace function on $\Delta$ is the
function $f_\Delta(x)$ such that $\mu_\Delta(s_1) = f_\Delta(s_1)\,ds_1$.
Either from the exact form of Corollary~\ref{corollary:haar_delta},
$$
f_\Delta(x) = \frac{1}{8\pi}\sqrt{16-x^2},
$$
or the above series expansion, we can graph the density function $f_\Delta(x):$
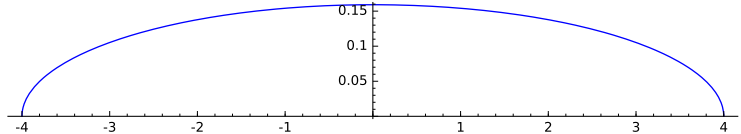
\begin{center}
\scalebox{0.50}{\sageplot{plot_f_SU2_diag}}
\end{center}
and its cumultative density function $F_\Delta(x) = \int_0^x f_\Delta(t)\,dt$:
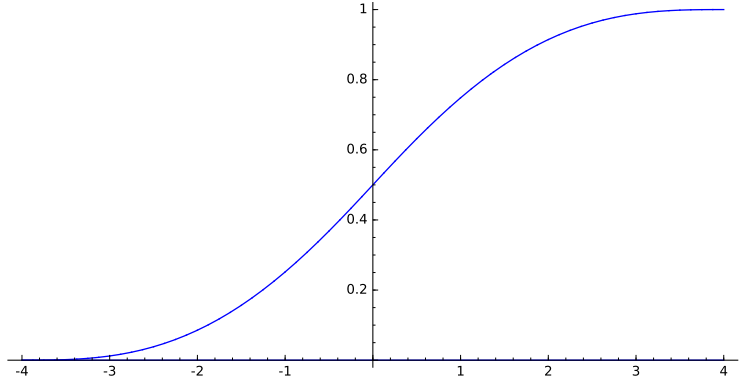
\begin{center}
\scalebox{0.50}{\sageplot{plot_F_SU2_diag}}
\end{center}
In the next section we recall results of Gilles Lachaud giving exact functional
expressions for the density functions of the trace on the groups $G$ and $H$.

%%%%%%%%%%%%%%%%%%%%%%%%%%%%%%%%%%%%%%%%%%%%%%%%%%%%%%%%%%%%%%%%%%%%%%%%%%%%%%%%
%%%%%%%%%%%%%%%%%%%%%%%%%%%%%%%%%%%%%%%%%%%%%%%%%%%%%%%%%%%%%%%%%%%%%%%%%%%%%%%%
\section{On results of Lachaud}
\label{section:lachaud}
%%%%%%%%%%%%%%%%%%%%%%%%%%%%%%%%%%%%%%%%%%%%%%%%%%%%%%%%%%%%%%%%%%%%%%%%%%%%%%%%
%%%%%%%%%%%%%%%%%%%%%%%%%%%%%%%%%%%%%%%%%%%%%%%%%%%%%%%%%%%%%%%%%%%%%%%%%%%%%%%%

In this section we detail results of Gilles Lachaud concerning the global
form of the trace function for the groups $G$ and $H$, growing out of his
interest in the theory of compact Lie groups and discussions with the
authors during the joint supervision of the Ph.D.\/ project of the second
author. His work gives an alternative approach to develop the local expansions
for the trace in the vicinity of $s_1 = \pm 4$.

We highlight certain main results in the presentation~\cite{Lachaud-ATI}
and preprint~\cite{Lachaud-SU2xSU2} of Lachaud.  Following the notation
of this work, we denote by $f_G(x)$ and $f_H(x)$ the distribution functions
for the trace on $G$ and $H$ such that
$$
\mu_G(s_1) = f_G(s_1) ds_1 \mbox{ and } \mu_H(s_1) = f_H(s_1) ds_1.
$$
The first main result is an exact form for the distribution function
for $G = \USp4$.

\begin{theorem}[Lachaud~\cite{Lachaud-ATI}]
\label{theorem:Lachaud-USp4}
For $G = \USp4$, if $|x| < 4$, the distribution function $f_G$ of the trace function
is given by
$$
f_G(x) = -\frac{64}{15\pi}\sqrt{|x|}\left(1 - \frac{x^2}{16}\right)^2
  \mathcal{P}_{\frac{1}{2}}^2\left(\frac{x^2 + 16}{4x}\right)
$$
where
$$
\mathcal{P}_{b}^a(z) = \frac{\Gamma(a+b+1)}{2\pi\Gamma(b+1)}
  \int_0^{2\pi} \left(z + \sqrt{z^2-1} \cos(\varphi)\right)^b \cos(a\varphi)d\varphi.
$$
\end{theorem}

\noindent{\bf Remark.}
In particular for $z = (x^2+16)/4x$ and $(a,b) = (2,1/2)$, we have
$$
\mathcal{P}_{\frac{1}{2}}^2(z) = \frac{15}{8\pi}
  \int_0^{2\pi} \left(\frac{x^2+16 + \sqrt{(x^2 + 16)^2 - 16x^2} \cos(\varphi)}{4x}\right)^{1/2} \cos(2\varphi)d\varphi
$$
As a corollary, Lachaud finds an alternative expression for the trace distribution
$f_G$ in terms of the elliptic integrals of the first kind
$$
K(m) = \int_0^{\pi/2} \frac{d\varphi}{\sqrt{1-m\sin^2(\varphi)}}
     = \int_0^1 \frac{du}{\sqrt{(1-u^2)(1-mu^2)}}
     = \frac{\pi}{2}\,{}_2F_1\!\left(\frac{1}{2},\frac{1}{2};1;m\right)\!\ccomma
$$
and of the second kind:
$$
E(m) = \int_0^{\pi/2} \sqrt{1-m\sin^2(\varphi)}\,d\varphi
     = \int_0^1 \sqrt{\frac{1-mu^2}{1-u^2}}\,du
     = \frac{\pi}{2}\,{}_2F_1\!\left(\frac{1}{2},-\frac{1}{2};1;m\right)\!\cdot
$$
\begin{corollary}[Lachaud~\cite{Lachaud-ATI}]
For $G = \USp4$, the distribution function $f_G: [-4,4] \rightarrow \RR$
of the trace function is given by
$$
f_G(x) = \frac{64}{15\pi}\big( (m^2 - 16m + 16)E(m) - 8(m^2 - 3m + 2)K(m) \big),
$$
where $m = 1 - x^2/16$.
\end{corollary}

It should be noted that the functions $K(m)$ and $E(m)$ admit well-known power series
representations in $m$:
$$
K(m) = \frac{\pi}{2} \sum_{n=0}^\infty \left( \frac{(2n)!}{4^n(n!)^2} \right)^{\!2}\!  m^n \mbox{ and }
E(m) = \frac{\pi}{2} \sum_{n=0}^\infty \left( \frac{(2n)!}{4^n(n!)^2} \right)^{\!2}\! \frac{m^n}{1-2n}
$$
which permits one to compute a power series representation in $m = 1-x^2/16$, from
which one can derive series expansions around $x = \pm 4$.  From the global form for
$f_G(x)$, one can compute a graphic representation for the density function:
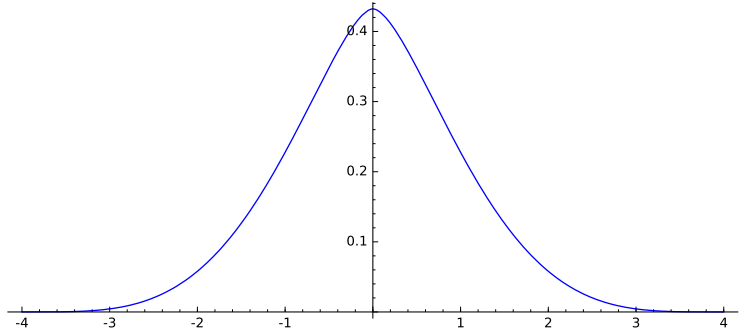
\begin{center}
\scalebox{0.50}{\sageplot{plot_f_USp4}}
\end{center}
and its cumulative density function $F_G(x) = \int_0^x f_G(x) dx$:
\begin{center}
\scalebox{0.50}{\sageplot{plot_F_USp4}}
\end{center}

\begin{theorem}[Lachaud~\cite{Lachaud-SU2xSU2}]
\label{theorem:Lachaud-SU2xSU2}
For $H = \SUxSU2$, the density function $f_H:[-4,4] \rightarrow \RR$
of the trace function is given by
$$
f_H(x) = \frac{m^2}{2\pi}\,{}_2F_1\!\left(\frac{1}{2},\frac{3}{2};3;m\right)\!\ccomma
$$
where ${}_2F_1(a,b;c;z)$ is the hypergeometric function and $m = 1 - x^2/16$.
\end{theorem}

\noindent
From the explicit formula for the $f_H(x)$, we can compute its graphical representation:
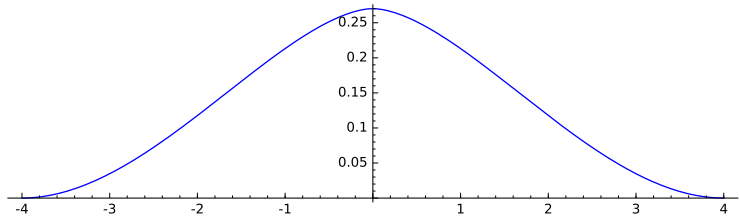
\begin{center}
\scalebox{0.50}{\sageplot{plot_f_SU2_square}}
\end{center}
and its cumulative density function $F_G(x) = \int_0^x f_H(x) dx$:
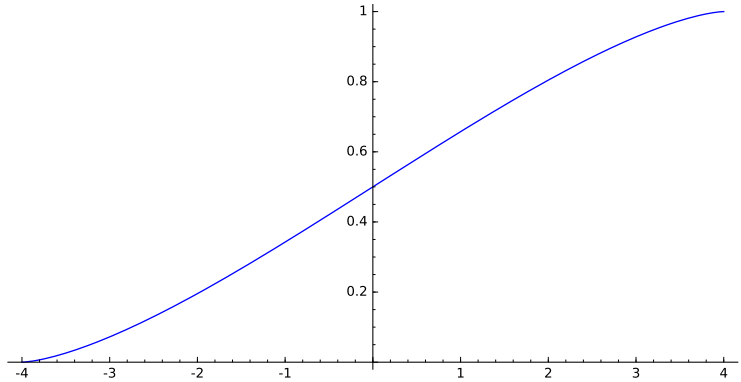
\begin{center}
\scalebox{0.50}{\sageplot{plot_F_SU2_square}}
\end{center}

These results of Gilles Lachaud, in particular Theorems~\ref{theorem:Lachaud-USp4}
and~\ref{theorem:Lachaud-SU2xSU2}, give global expressions for the local power
series expressions for $\mu_G = f_G ds_1$ and $\mu_H = f_H ds_1$ in
Theorem~\ref{theorem:haar_GH_eps}, in terms of classical special functions.
Both approaches give exact convergent expressions for the density functions
$f_G$ and $f_H$ and cumulative density functions $F_G$ and $F_H$ for $s_1$
in the interval $[-4,4]$, or equivalently $\varepsilon$ in $[0,8]$.
In the next section we describe the application to a problem of extremal trace
distributions in families which in part motived each of Gilles and the authors
to carry out these respective computations of the trace distributions.

%%%%%%%%%%%%%%%%%%%%%%%%%%%%%%%%%%%%%%%%%%%%%%%%%%%%%%%%%%%%%%%%%%%%%%%%%%%%%%%%
%%%%%%%%%%%%%%%%%%%%%%%%%%%%%%%%%%%%%%%%%%%%%%%%%%%%%%%%%%%%%%%%%%%%%%%%%%%%%%%%
\section{Distribution of extremal traces in families}
\label{section:extremal_traces}
%%%%%%%%%%%%%%%%%%%%%%%%%%%%%%%%%%%%%%%%%%%%%%%%%%%%%%%%%%%%%%%%%%%%%%%%%%%%%%%%
%%%%%%%%%%%%%%%%%%%%%%%%%%%%%%%%%%%%%%%%%%%%%%%%%%%%%%%%%%%%%%%%%%%%%%%%%%%%%%%%

% This might form part of the introduction:
% \noindent{\bf General introduction:}\vspace{1mm} % COMMENT

Let $\cM_2$ be the moduli space of genus-$2$ curves, $\cH_D$ the Humbert surface
of genus-$2$ curves with RM by a real quadratic order of discriminant $D$ and
$\cX^N$ a Shimura curve of genus-$2$ curves with QM by an order of discriminant $N$
in which $D$ is inert or ramified. Achter and Howe~\cite{Achter-Howe} determine
bounds on the density of split abelian surfaces among all principally
polarized abelian surfaces over a finite field of $q$ elements.  Such split surfaces
are accounted for by points on Humbert surfaces with $D$ a square.  We are interested
in the asymptotic contribution of RM and QM points among abelian surfaces, such that
that the trace of Frobenius is close to the ends of the Weil interval $[-2\sqrt{q},2\sqrt{q}]$.
More generally, associated to the inclusions of moduli spaces
$
\cX^N \hookrightarrow \cH_D \hookrightarrow \cM_2
$
we pose the question of the asymptotic contribution to extremal traces of each moduli
space and the relative densities.

%\vspace{1mm}
%\noindent{\bf Transition to Sato--Tate groups:}\vspace{1mm} % COMMENT

The respective vertical Sato--Tate groups associated with generic endomorphism ring
is $G = \USp4$, for real multiplication is $H = \SUxSU2$ and for quaternionic
multiplication is $\Delta \isom \SU2 \subset H$.  A family of genus-$2$ curves with
parametrized RM or QM endomorphism ring structure, whose base les over a Humbert
surface or Shimura curve, will have trace distribution characteristic of $H$ or
$\Delta$.  As the respective trace distribution show (see
Corollary~\ref{corollary:haar_delta}, Theorems~\ref{theorem:haar_GH_eps},
\ref{theorem:Lachaud-USp4} and~\ref{theorem:Lachaud-SU2xSU2}),
such families tend to have a higher density of extremal traces.

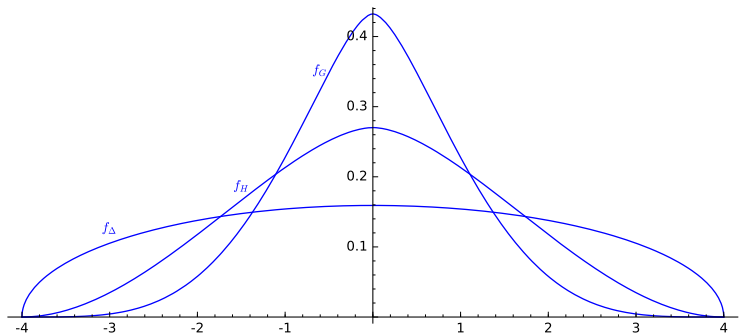
\begin{center}
\scalebox{0.50}{\sageplot{plot_dists}}
\end{center}

As a consequence, the cumulative density function for the trace of Frobenius in
the neighborhood of $-4\sqrt{q}$, grows faster for $\Delta$ than for $H$ and
faster for~$H$ than for the generic group~$G$.

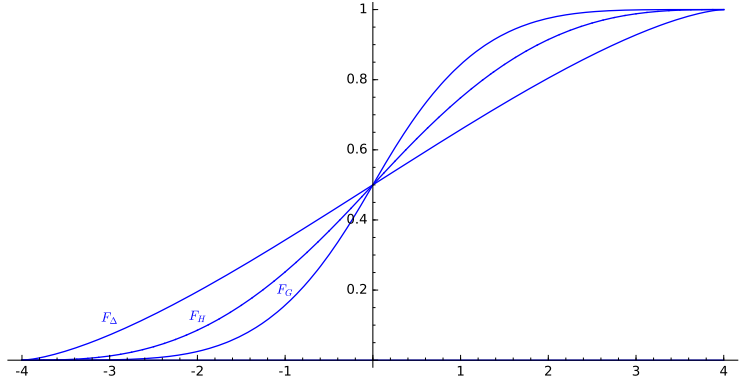
\begin{center}
\scalebox{0.50}{\sageplot{plot_cumul_dists}}
\end{center}

Specifically, in the following table, we summarize the first order approximations
to the growth of $\eps = s_1 + 4$ in neighborhood of $\eps = 0$ from
Theorems~\ref{theorem:haar_GH_eps_approx} and \ref{theorem:haar_delta_eps},
together with the associated orders of magnitude of points on the associated moduli spaces.
$$
\begin{array}{|l|l|l|}
\hline &  & \\[-4mm]
\displaystyle
f_G(\eps) \approx \frac{\eps^4}{64\pi} &
\displaystyle
F_G(\eps) \approx \frac{\eps^5}{320\pi} &
\big|\cM_2(\FF_q)\big| \approx q^3
\\[4mm] \hline &  & \\[-4mm]
\displaystyle
f_H(\eps) \approx \frac{\eps^2}{8\pi} &
\displaystyle
F_H(\eps) \approx \frac{\eps^3}{24\pi} &
\big|\cH_D(\FF_q)\big| \approx q^2
\\[4mm] \hline & & \\[-4mm]
\displaystyle
f_\Delta(\eps) \approx \frac{\eps^{1/2}}{\sqrt{8}\pi} &
\displaystyle
F_\Delta(\eps) \approx \frac{\eps^{3/2}}{3\sqrt{2}\pi} &
\big|\cX^N(\FF_q)\big| \approx q
\\[4mm] \hline
\end{array}
$$
Setting $\eps = d/\sqrt{q}$, where $d$ is the defect from the border $t = -4\sqrt{q}$
of the Weil interval, we observe the relative growth of the cumulative densities:
$$
\frac{F_\Delta(\eps)}{F_H(\eps)} = \frac{4\sqrt{2}}{d^{3/2}}\,q^{3/4} \mbox{ and }
\frac{F_H(\eps)}{F_G(\eps)} = \frac{40}{3d^2}\,q.
$$
Consequently a search of random points on a Shimura curve is more likely to exhibit
curves of extremal trace than on a Humbert surface or for a generic curve. Every
genus-$2$ curve with QM, however, has split Jacobian over a finite field.
Similarly, a search of random points on a Humbert surface is more likely to exhibit
curves of extremal trace than for a generic curve, and a general curve with moduli
on a Humbert surface has absolutely commutative endomorphism ring and nonsplit Jacobian.

Considering all curves over such spaces, we find that in the neighborhood of
$\eps = 0$, the rational points $\cH_D(\FF_q)$ of any single Humbert surface
dominates the extremal traces within the full set of points $\cM_2(\FF_q)$ on
the full moduli space.  Specifically, with $\eps = d/\sqrt{q}$, as above,
we have
$$
\begin{array}{|r@{\;}c@{\;}c@{\;}c@{\;}c|}
\hline &  &  &  & \\[-4mm]
\displaystyle
F_G(\eps) \big|\cM_2(\FF_q)\big| & \approx &
\displaystyle
\frac{\eps^5}{320\pi} q^3  & = &
\displaystyle
\frac{d^5}{320\pi} \sqrt{q}
\\[4mm] \hline &  &  &  & \\[-4mm]
\displaystyle
F_H(\eps) \big|\cH_D(\FF_q)\big| & \approx &
\displaystyle
\frac{\eps^3}{24\pi} q^2  & = &
\displaystyle
\frac{d^3}{24\pi} \sqrt{q}
\\[4mm] \hline &  &  &  & \\[-4mm]
\displaystyle
F_\Delta(\eps) \big|\cX^N(\FF_q)\big| & \approx &
\displaystyle
\frac{\eps^{3/2}}{3\sqrt{2}\pi} q & = &
\displaystyle
\frac{d^{3/2}}{3\sqrt{2}\pi} \sqrt[4]{q}
\\[4mm] \hline
\end{array}
$$
Since $1/320\pi < 1/24\pi$, this gives an apparent paradox for any defect
$d < \sqrt{40/3} < 3.6515$ from the end of the Weil interval.  This shows
clearly that the vertical Sato--Tate distribution for the trace of Frobenius
on $[-4,4]$, as a limit in $q$, does not allow one to conclude existence
of points in short intervals (of length $O(1/\sqrt{q})$).  Moreover, for
particular real discriminant $D$, we prove a proposition concerning the
discrete structure of moduli points in $\cH_D(\FF_q)$, which shows the
dependence of the arithmetic of the Galois representations of Frobenius
with the geometry of $\cH_D$.

\begin{proposition}
\label{proposition:exclusion_zone}
Suppose $A/\FF_q$ is a Jacobian surface such that $\End(A)$ is commutative.
Then there exists a unique $D$ such that the moduli point of $A$ lies in
$\cH_D(\FF_q) \subset \cM_2(\FF_q)$, and the associated point $(s_1,\delta_0)$
in $\fD_2$ satisfies
$$
\delta_0 = m_0\sqrt{\frac{D}{q}} \le \eps,
$$
where $\eps = 4 - |s_1|$ and $m_0 > 0$ is an integer.  In particular,
the defect $d = \eps\sqrt{q}$ with respect to the Weil bound satisfies
$\sqrt{D} \le d$.
\end{proposition}

\begin{proof}
Since $\End(A)$ is commutative, either $A$ is simple and $\End(A)$
is  an order in a CM field, or $A$ is isogenous to a product of
nonisogenous elliptic curves and $\End(A)$ is a suborder of a product
of imaginary quadratic orders.
Let $\End^s(A)$ be the subring fixed by the Rosatti involution,
containing $\ZZ[\pi+\bar{\pi}]$.  Then $\End^s(A)$ is a quadratic order
of discriminant $D$.
If $A$ is simple, then $\End^s(A)$ is a real quadratic order, and otherwise
$\End(A)^s$ is an order the \'etale algebra $\QQ \times \QQ$ of discriminant
$D = m^2$.  The subring $\ZZ[\pi+\bar{\pi}]$ is then a suborder of index
$m_0$ and discrimnant $m_0^2 D$. In view of the optimal embedding of
$\End^s(A)$ in $\End(A)$, the moduli point of $A$ lies in $\cH_D(\FF_q)$.
By definition
$
\delta_0^2 q = m_0^2 D = \disc(\ZZ[\pi+\bar{\pi}]),
$
and the bounds follow since $(s_1,\delta_0)$ is a point in $\fD_2$.
\end{proof}

\noindent{\bf Remark.}
The condition that $\End(A)$ is commutative includes all absolutely simple
abelian surfaces, as well as those isogenous to a product $E_1 \times E_2$
such that $E_1$ and $E_2$ are not isogenous.  The latter products arise from
abelian surfaces whose moduli are on Humbert surfaces of square discriminants.
In contrast, any abelian surface $A$ with quaternionic multiplication by
a maximal order $\cO_B$ in a quaternion algebra $B$ of discriminant $N$,
whose moduli point lies on a Shimura curve $X^N$, will have larger
endomorphism ring over the algebraic closure. Over an extension such an
abelian surface is isogenous to $E\times E$, where $\cO_K = \End(E)$ is the maximal
order in an imaginary quadratic field.  In terms of the domain $\fD_2$, for
any field $k/\FF_q$ such that $\End_k(A) = \End(A)$, the normalized Frobenius
characteristic polynomial is of the form $(x^2 - tx + 1)^2$, and the
associated point $(s_1,\delta_0)$ satisfies $\delta_0 = 0$.  Since infinitely
many real quadratic orders embed in $\cO \subset \End(A) \subseteq \MM_2(\cO_K)$,
the Shimura curve lies in the intersection of infinitely many Humbert surfaces $\cH_D$.

\begin{corollary}
\label{corollary:exclusion_disc_bound}
A Jacobian surface $A/\FF_q$ with absolutely commutative endomorphism ring
and defect $d = 4\sqrt{q} - |a_1|$ with respect to the Weil bound has moduli
in a unique Humbert surface $\cH_D(\FF_q)$ with $\sqrt{D} \le d$.
\end{corollary}

Proposition~\ref{proposition:exclusion_zone} asserts that the points $(s_1,\delta_0)$
in family over $\cH_D$ lie in discrete bands, and consequently, there is a zone of
exclusion : for $\delta_0 \ne 0$ then %$\delta_0 \ge \sqrt{D/q}$.
%However, for given $D$ there exists a zone of exclusion
$$
\delta_0^2 = s_1^2 - 4s_2 = m_0^2 \frac{D}{q} \ge \frac{D}{q}\cdot
$$
Moreover, the dense set of moduli points of Jacobian surfaces with absolutely
commutative endomorphism ring lie in disjoint sheets $\cH_D$ (which intersect
pairwise in unions of the Shimura curves of codimension $1$ whose associated
Frobenius representations constribute to the domain $\delta_0 = 0$).
This gives a stratification of $\cM_2$ into Humbert surfaces.  The Sato--Tate
distribution of $\USp4$ (over the generic space $\cM_2$) can be visualized
as layered in the axis $D$ by its contributions of Sato--Tate distributions
for $\SUxSU2$ with zone of exclusion $\delta_0^2 \ge D/q$.

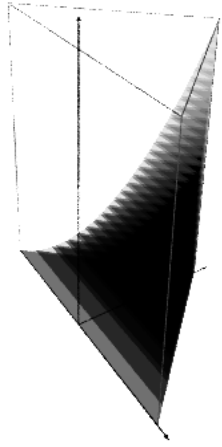
\begin{center}
%\scalebox{0.50}{\sageplot{M2_strata}}
\scalebox{0.50}{\includegraphics{img/m2_strata-grad.png}}
\end{center}

This stratification of the moduli space $\cM_2$ by Humbert surfaces $\cH_D$
is reflected in the precise algebraic relation between the Haar measures
$\mu_G$ and $\mu_H$,
$$
\frac{\mu_G(t_1,t_2)}{\mu_H(t_1,t_2)}
  = \frac{\mu_G(s_1,s_2)}{\mu_H(s_1,s_2)}
  = \frac{\mu_G(s_1,\delta_0)}{\mu_H(s_1,\delta_0)}
  = \frac{D_0}{2}\ccomma
$$
a scalar multiple of the measure $\mu_H$, which measures the relative number of
strata $\cH_D$ at a given value of $D_0 = \delta_0^2$, contributing to the Haar
measure $\mu_G$ over $\cM_2$.

\vspace{2mm}
\noindent{\bf Example.} The fibration of $\cM_2/\FF_q$ into strata $\cH_D/\FF_q$,
and the discrete bands $\delta_0 = m_0 \sqrt{D/q}$ in $\cH_D(\FF_q)$ corresponding
to each index $m_0 = [\cO_D:\ZZ[\pi+\bar{\pi}]]$ can be visualised by plotting
the moduli points in $\cH_D(\FF_q)$.  For the prime $q = 47$, and discriminants
$D = 5$, $12$, $37$ and $97$, we have:

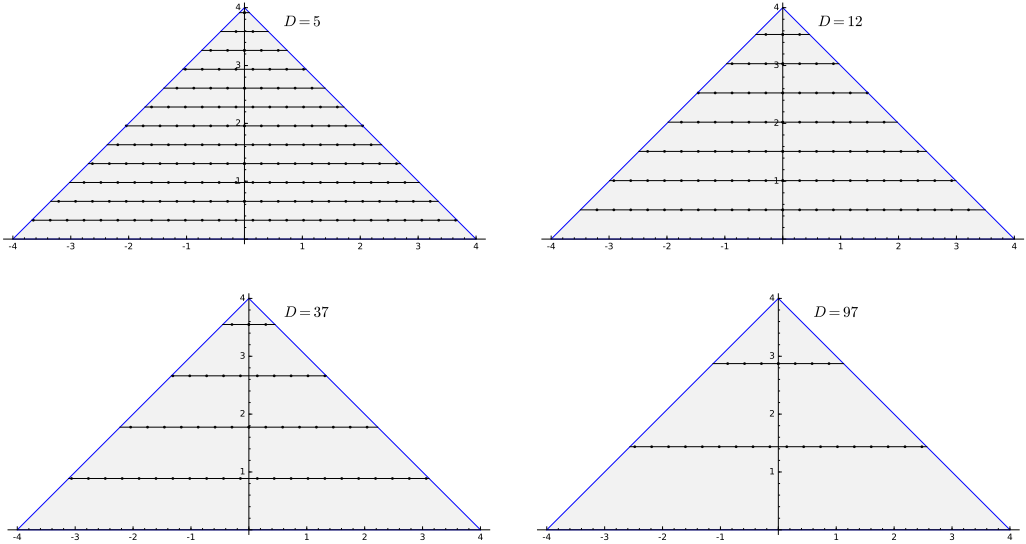
\begin{center}
\scalebox{0.33}{\sageplot{D2_ruled_05}} \quad
\scalebox{0.33}{\sageplot{D2_ruled_12}}
\end{center}

\begin{center}
\scalebox{0.33}{\sageplot{D2_ruled_37}}\quad
\scalebox{0.33}{\sageplot{D2_ruled_97}}
\end{center}

\section{Conclusion}

Based on the local expansions around the boundary of extremal traces ($s_1 = \pm4$)
for the degree-$4$ representations of $\USp4$, $\SUxSU2$ and $\SU2$, we derive
explicit Taylor series expansions in the neighborhood of the extremal trace
$s_1 = -4$ (and by symmetry around $s_1 = 4$).  In the application to curves
and Jacobians over finite fields, this corresponds to genus-$2$ curves with
many points (and with few points, respectively).
Noting that $\USp4$ is the generic Sato--Tate group for abelian surfaces, and
$\SUxSU2$ is the generic Sato--Tate group for abelian surfaces with RM, we
combine this quantitative result with a qualitative analysis of the
stratification of the moduli space $\cM_2$ (or $\cA_2$ of
principally polarized abelian surfaces) by Humbert surfaces parametrizing RM.
We use this to give a heuristic interpretation of the dominance of small RM
(including split abelian surfaces) among the Jacobians of curves with many points.

An interesting question remains the role of RM and more generally exceptional
endomorphisms in higher dimensions. From the Weyl integration formulas, one
derives similar expressions for the induced Haar measure on $\Sigma_g$.
As above, setting $t_j = 2\cos(\theta_j)$ and then
$$
D_0 = \prod_{j<k}(t_j-t_k)^2 \mbox{ and } D_1 = \prod_{j=1}^g(4-t_j^2),
$$
one finds for $G = \USp{2g}$,
$$
\mu_G(s_1,\dots,s_g) = \frac{1}{(2\pi)^g} \sqrt{D_0D_1}\,ds_1\cdots ds_g,
$$
and for $H = \SU2^g$,
$$
\mu_H(s_1,\dots,s_g) = \frac{g!}{(2\pi)^g} \frac{\sqrt{D_1}}{\sqrt{D_0}}\,ds_1\cdots ds_g.
$$
This gives an expression similar to that for $g = 2$:
\begin{equation}
\label{equation:ratio_G_to_H}
\frac{\mu_G(s_1,\dots,s_g)}{\mu_H(s_1,\dots,s_g)} = \frac{D_0}{g!}\cdot
\end{equation}
However, unlike in the case for $g = 2$, the single parameter $D = \disc(\cO)$
is not sufficient to uniquely characterize an order $\cO$ in a totally
real field or etale algebra.
Nevertheless, the Hilbert moduli space of a totally real order $\cO$ maps to
a $g$--dimensional subspace $\cH_g(\cO)$ of the $g(g+1)/2$--dimensional moduli
space $\cA_g$ of principally polarized abelian varieties of dimension $g$.
In general dimension we expect the ordinary points over $\FF_q$ to provide
the dominant contribution to the Haar measure, and such points partition
themselves into unique strata $\cH_g(\cO)$.  For abelian $3$--folds, this
gives a stratification of the $6$--dimensional space $\cA_3$ by $3$--folds
$\cH_3(\cO)$.

While the qualitative description is analogous, the interpretation of
the ratio~\eqref{equation:ratio_G_to_H} of Haar measures is not as obvious
since the discriminant $D$ doesn't characterize totally real orders of
rank $g > 2$.  Moreover the computation of an explicit Taylor series expansion
for the trace to quantify this qualitative contribution would be a much
more involved calculation, whose completion would nevertheless be interesting.

\ignore{
\section{Appendix}

Let $T = \SO2$, $K = \SU2$ and $T\times K = \SO2 \times \SU2$. The group $T$ arises
as the Sato--Tate group of a CM elliptic curve, which has a geometric moduli space of
dimension zero.  The product group $T \times K$ can arise as the Sato--Tate group

\begin{theorem}
The Haar measure on $\SO2 \times \SU2$ on $\fD_2$ is given by
$$
\frac{1}{64\pi^2} \sqrt{\frac{16 - (s_1 - \delta_0)^2}{16 - (s_1 + \delta_0)^2}}\, d\delta_0 ds_1
$$
\end{theorem}
}

\end{document}